\documentclass[10pt,a4paper]{amsart}
\usepackage[utf8]{inputenc}
\usepackage[pagebackref]{hyperref}
\usepackage{amssymb,color,esint}

\numberwithin{equation}{section}
\numberwithin{figure}{section}

\theoremstyle{plain}
\newtheorem{thm}{Theorem}
\newtheorem{lem}[thm]{Lemma}
\newtheorem{prop}[thm]{Proposition}
\newtheorem{coro}[thm]{Corollary}
\newtheorem*{claim}{Claim}

\theoremstyle{definition}
\newtheorem{example}[thm]{Example}
\newtheorem{definition}{Definition}

\theoremstyle{remark}
\newtheorem*{remark}{Remark}

\newcommand{\R}{\mathbb{R}}
\newcommand{\C}{\mathbb{C}}
\newcommand{\Z}{\mathbb{Z}}
\newcommand{\B}{\mathbb{B}}
\renewcommand{\P}{\mathbb{P}}
\newcommand{\tr}{\operatorname{Tr}}
\newcommand{\sinc}{\operatorname{sinc}}

\begin{document}

\title[Sampling real polynomials]{Sampling of real multivariate polynomials and 
pluripotential theory}

\author{Robert J. Berman}
\address{Department of Mathematical Sciences, 
Chalmers University of Technology and
University of Gothenburg, 412 96 G\"oteborg, Sweden}
\email{\href{mailto:robertb@chalmers.se}{\texttt{robertb@chalmers.se}}}

\author{Joaquim Ortega-Cerd\`a}
\address{Dept.\ Matem\`atica Aplicada i An\`alisi,
 Universitat  de Barcelona and BGSMath,
Gran Via 585, 08007 Bar\-ce\-lo\-na, Spain}
\email{\href{mailto:jortega@ub.edu}{\texttt{jortega@ub.edu}}}

\thanks{The first author was supported by grants from the European and Swedish 
Research Council and the Knut and Alice Wallenberg foundation and the second 
author is supported by the the Spanish Ministerio de Econom\'ia y Competividad 
(grant MTM2014-51834-P) and the Generalitat de Catalunya (grant 2014-SGR-289)}

\begin{abstract}
We consider the problem of stable sampling of multivariate real polynomials of
large degree  in a general framework where the polynomials are defined on an
affine real algebraic variety $M$, equipped with a weighted measure. In
particular, this framework contains the well-known setting of trigonometric
polynomials (when $M$ is a torus equipped with its invariant measure), where the
limit of large degree corresponds to a high frequency limit, as well as the
classical setting of one-variable orthogonal algebraic polynomials (when $M$ is
the real line equipped with a suitable measure), where the sampling nodes can be
seen as generalizations of the zeros of the corresponding orthogonal
polynomials. It is shown that a necessary condition for sampling, in the general
setting, is that the asymptotic density of the sampling points is greater than
the density of the corresponding weighted equilibrium measure of $M$, as defined
in pluripotential theory. This result thus generalizes the well-known Landau
type results for sampling on the torus, where the corresponding critical density
corresponds to the Nyqvist rate, as well as the classical result saying that the
zeros of orthogonal polynomials become equidistributed with respect to the
logarithmic equilibrium measure, as the degree tends to infinity. 
\end{abstract}
\date{\today}
\maketitle

\section{Introduction}
\subsection{Background}
By the classical Whittaker-Shannon-Kotelnikov sampling theorem a band-limited
signal $f$ on the real line $\R$, normalized so that its frequency is in
$[-1,1]$ may be recovered from its values at the points $t_{j}=j\pi$ where $j$
ranges over the integers and 
\[
\int_{\R}\left|f(t)\right|^{2}dt=\pi\sum_{j}\left|f(t_{j})\right|^{2}.
\]
In mathematical terms,  $f$ is in the Paley-Wiener space $PW_{1}(\R)$ consisting
of all functions in $L^{2}(\R)$ whose Fourier transform is supported in
$[-1,1]$. More generally, in the theory of \emph{non-regular sampling} a
sequence $\Lambda:=\{\lambda\}_{\lambda\in\Lambda}$ of points on the real line
$\R$ is said to be \emph{sampling} for $PW_{1}(\R)$ if there exists a constant
$C$ such that the following sampling inequality holds 
\[
\frac{1}{C}\int_{\R}\left|f(t)\right|^{2}dt\leq\sum_{\lambda\in\Lambda}
\left|f(\lambda)\right|^{2}\leq C\int_{\R}\left|f(t)\right|^{2}dt
\]
for any any $f\in PW_{1}(\R)$, ensuring that the reconstruction of $f$ is stable 
in the $L^{2}$-sense. Corresponding results also hold in the higher dimensional 
setting where $\R$ is replaced with $\R^{n}$ and the band $[-1,1]$ with the 
unit-cube $[-1,1]^{n}$ (or more general any fixed convex body of volume one). By 
the seminal result of Landau \cite{la}, a necessary condition for a set 
$\Lambda$ to be sampling is that the corresponding asymptotic density of points 
in $\R^{n}$ (in the sense of Beurling) is at least equal to the Nyqvist rate 
$1/\pi^{n}$,
i.e. 
\[
\liminf_{R\rightarrow\infty}\frac{\#\{\Lambda\cap R\Omega\}}{R^{n}}
\geq\int_{\Omega}\frac{1}{\pi^{n}}dt
\]
(uniformly over translations) for any smooth domain $\Omega \subset \R^{n}$
assuming a uniform separation lower bound on the points in $\Lambda$. In
one-dimension the reversed strict inequality is also a sufficient condition for
sampling, but not in higher dimensions. 

By a rescaling, Landau's density results may also be formulated in terms of the
high frequency limit which appears when the frequency domain $[-1,1]^{n}$ is
replaced with $k[-1,1]^{n}$ for $k$ large, i.e. $PW_{1}(\R^{n})$ is replaced
with the corresponding Paley-Wiener space $PW_{k}(\R^{n})$. In this context a
sequence $\Lambda_{k}:=\{\lambda^{(k)}\}$ of sets of points on $\R^{n}$ is said
to be sampling for $PW_{k}(\R^{n})$ if 
\[
\frac{1}{C}\int_{\R^{n}}\left|f(t)\right|^{2}dt\leq\frac{1}{k^{n}} 
\sum_{\lambda^{(k)}\in\Lambda_{k}}\left|f(\lambda^{(k)})\right|^{2}\leq 
C\int_{\R^{n}}\left|f(t)\right|^{2}dt
\]
for any $f\in PW_{k}(\R^{n})$ with the constant $C$ independent of $k$. For the
sake of simplicity if it is clear from the context we will omit the superindex
$k$ in $\lambda^{(k)}$ and write simply $\lambda\in \Lambda_k$. The
corresponding necessary density condition on the sampling points may then be
reformulated as 
\[
\liminf_{R\rightarrow\infty}\frac{\#\{\Lambda_{k}\cap\Omega\}}{k^{n}}\geq 
\int_{\Omega}\frac{1}{\pi^{n}}dt
\]
uniformly over translations for any domain $\Omega \subset \R^{n}$ with
$|\partial \Omega|=0$. (Landau's setting corresponds to the case when
$\Lambda_{k}$ is of the form $k^{-1}\Lambda$, but his arguments extend to this
high-frequency setting).

There is also a natural \emph{compact} analogue of the Paley-Wiener setting on
$\R^{n}$ which is the one which is most relevant for the present paper, where
$\R$ (or $\R^{n})$ is replaced with the circle $S^{1}:=\R/2\pi$. Then the role
of $PW_{k}(\R)$ is played by the space $H_{k}(S^{1})$ of all finite Fourier
series on $[0,2\pi]$ with frequencies in $[-k,k]$, i.e. the space of all
trigonometric polynomials of degree at most $k$. A sampling sequence of finite
sets of points $\Lambda_{k}\subset S^{1}$ in this setting is also called a
Marcinkiewicz-Zygmund family \cite{o-s}. In a similar setting to ours it has
been studied in \cite{bb} the sampling sequences for powers of a line bundle
with positive curvature on compact complex manifolds. In this case there are
precise estimates for the Bergman kernel so that Landau's techniques carry
through.

From an abstract point of view the previous settings fit into a general Hilbert 
space framework where $H_{k}(M)$ is a given sequence of Hilbert spaces of 
functions on a set $M$ with reproducing kernels $K_{k}(x,y)$. Then a sequence 
$\Lambda_{k}$ of sets of points on $M$ is said to be\emph{ sampling for 
$H_{k}(M)$} if the family of normalized functions $\kappa_{\lambda} 
:=K_{k}(\cdot,\lambda)/\|K_{k}(\cdot,\lambda)\|$ for $\lambda\in\Lambda_{k}$, 
form a \emph{frame} in the Hilbert space $H_{k}(M)$, in the sense of 
Duffin-Schaeffer \cite{du}, i.e.:
\[
\frac{1}{C}\left\Vert f\right\Vert ^{2}\le \sum_{\lambda\in\Lambda_k} |\langle 
f,\kappa_{\lambda}\rangle|^2\le  C\left\Vert f\right\Vert ^{2},\quad \forall 
f\in
H_k(M),
\]
which is equivalent to the sampling inequalities: 
\begin{equation}\label{sampineq}
\frac{1}{C}\left\Vert f\right\Vert 
^{2}\leq \sum_{\lambda\in\Lambda_k}
\frac{|f(\lambda)|^2}{K_k(\lambda,\lambda)} \leq C\left\Vert f\right\Vert ^{2},
\quad \forall f\in H_k(M).
\end{equation}
where we will assume that $C$ can be taken to be independent of $k$.

\subsection{The present setting}
The main aim of the present paper is to generalize the Landau type necessary
density conditions for sampling on $S^1$ to a general setting where the Hilbert
space $H_{k}(M)$ consist of polynomials of degree at most $k$ on an affine real
algebraic variety $M$ equipped with a weighted measure. We are not dealing with 
the very interesting problem of finding sufficient conditions for sampling 
multivariate polynomials. This and its numerical implementation is a very basic 
question in signal analysis, see for instance \cite{fgs} and the references 
therein for the one-variable numerical sampling.

Our setting is the following: by
definition $M$ is the variety cut out by a finite numbers of polynomials on
$\R^{m}$ and $H_{k}(M)$ is the space of polynomials of total degree at most $k$
restricted to $M$ and equipped with the $L^{2}$ norm
\[
\left\Vert p_{k}\right\Vert _{L^{2}(e^{-k\phi}\mu)}^{2}:=\int_{M}\left|p_{k} 
\right|^{2}e^{-k\phi}d\mu
\]
defined by a compactly supported measure $\mu$ on $M$ and a continuous function 
$\phi$ on $M$ (referred to as the weight function). Following \cite{bbw} we will 
refer to the pair $(\mu,\phi)$ as a "weighted measure". In order that the latter 
norm be non-degenerate some regularity assumption has to be made on $\mu$. The 
affine case, i.e. when $M=\R^m$  is thus the classical setting for 
multivariate orthogonal polynomials. We will assume two regularity conditions: 
the Bernstein-Markov property and moderate growth, see Section~\ref{setup} for 
the precise definitions.

Our first main result in this general setting is
\begin{thm}\label{general}
Let $M$ be an affine real algebraic variety equipped with a non degenerate
measure $\mu$ and a weight function $\phi$. Assume that the pair ($\mu$, $\phi$)
satisfies the Bernstein-Markov property \eqref{BM}  and it is of moderate growth
\eqref{MG}. Then a necessary condition for a sequence $\Lambda_{k}$ of sets of
points in $M$ to be sampling for the space $H_{k}(M)$ of polynomials of degree
at most $k$, with respect to the weight $k\phi$ and measure $\mu$, is that 
\begin{equation}\label{landausamp}
\liminf_{k\rightarrow\infty}\frac{1}{N_{k}}\sum_{\lambda
\in \Lambda_{k}}\delta_{\lambda}\geq\mu_{eq} 
\end{equation} 
in the weak topology on the measures on $M$, where $\mu_{eq}$ denotes the
normalized equilibrium measure of the weighted measure $(\mu,\phi)$ and
$N_k=\dim(H_k(M))$.
\end{thm}

\subsection{Sampling on compact real algebraic varieties equipped with a volume
form}

One disadvantage of the definition of the sampling inequalities \eqref{sampineq}
in this general setting is that it is of a rather abstract nature as it involves
the reproducing kernel $K_{k}(x,x)$ which in general is impossible to compute
explicitly. On the other hand, only the asymptotic behaviour of $K_{k}(x,x)$ as
$k\rightarrow\infty$ is needed and these asymptotics can often be estimated.
Also, if $\mu_{eq}$ is absolutely continuous with respect to Lebesgue measure
than the condition \eqref{landausamp} above may be written as
\begin{equation}\label{Nyquist}
\liminf_{k\rightarrow\infty}\frac{\#\{\Lambda_{k}\cap\Omega\}}{\#N_{k}} 
\geq\frac{\mu_{eq}(\Omega)}{\mu_{eq}(M)}
\end{equation}
for any smooth domain $\Omega$.We will refer to the latter condition as the 
"pluripotential Nyqvist bound". One particularly interesting case where 
Theorem~\ref{general} can be made explicit is the following:

\begin{thm}\label{thm:sampling affine compact intro}
Let $M$ be an $n$-dimensional affine real algebraic variety, which is
non-singular and compact, let $\mu$ be a volume form on $M$ and let $\phi=0$.
Then there exists a positive constant $C$ such that the reproducing kernel for
$(H_{k}(M),\mu)$ satisfies 
\begin{equation}\label{eq:bounds on bergman function intro}
\frac{1}{C}k^{n}\leq K_{k}(x,x)\leq Ck^{n}
\end{equation}
and thus $(\mu,\phi)$ is non degenerate, it satisfies the Bernstein-Markov
property and it is of moderate growth. Moreover, a necessary condition for a
sequence $\Lambda_{k}$ of sets of points on $X$ to be sampling for $H_k(M)$ is
that the density of sampling points is at least equal to the density of the
equilibrium measure $\mu_{eq}$ of $M$, as $k\rightarrow\infty$, i.e., the 
pluripotential Nyquist bound \eqref{Nyquist} holds.
\end{thm}

The definition of the equilibrium measure of M and more generally the 
equilibrium measure attached to a weighted measure will be recalled in 
Section~\ref{extremal}. As pointed out above this result thus generalizes the 
results in \cite{o-s} concerning the case when $M$ is the unit-circle. Moreover, 
the case when $M$ is the unit-sphere corresponds to the case studied in 
\cite{Ma}, where the signals in questions are spherical harmonics. While in all 
these special cases the equilibrium measure $\mu_{eq}$ is explicitly given by 
the Haar measure (since the corresponding Riemannian manifolds are homogeneous) 
the equilibrium measure of a general real affine algebraic variety appears to be 
of a highly non-explicit nature. Another generalization of the homogeneous cases 
was considered in \cite{op}, where the signals are ``band-limited'' sums of 
eigenfunctions of the Laplacian on a given compact Riemannian manifolds $(M,g)$ 
and then the role of the equilibrium measure is played by the Riemannian volume 
form. 

It is not evident a priori that there are sampling sequences at all. This 
is assured with the following Bernstein type theorem:
\begin{thm}\label{bernsteincomplet}
 Given a smooth compact real manifold $M\subset \mathbb \R^m$ of dimension $n$,
the following are equivalent:
 \begin{itemize}
  \item $M$ is algebraic
  \item $M$ satisfies a Bernstein inequality, i.e., for some $q\ge 1$ (or for
all $q\ge 1$):
  \[
    \|\nabla_t p\|_{L^q(M)} \le C_q \deg(p) \|p\|_{L^q(M)}.
  \]
  \item There is a uniformly separated  $\Lambda_k$ such that for some (all) 
$q\ge 1$
  \[
    \int_M |p|^q dV_M \lesssim \frac 1{k^n}\sum_{\lambda\in\Lambda_k}
|p(\lambda)|^q \lesssim \int_M |p|^q dV_M, \quad \forall p\in \mathcal
P_k(\R^m).
  \]
 \end{itemize}
\end{thm}
This generalizes the main result of \cite{blmr} where the case $q=\infty$ was 
considered.

\subsection{Sampling of multivariate real polynomials on 
convex domains}
Another instance where Theorem~\ref{general} can be made more precise is the 
case were $M=\R^n$,  $\mu$ is the Lebesgue measure restricted to a smooth 
bounded convex domain $\Omega$ and $\phi=0$. In this case the equilibrium 
measure is very well understood, see \cite{BedTay} and \cite{blmr}. It behaves 
roughly as $d\mu_{eq}\simeq 1/\sqrt{d(x,\partial\Omega)}dV$, (this will also 
follow from the asymptotics \eqref{kernelconvex} below). 
\begin{thm}\label{thm:samplingconvex}
Let $\Omega$ be a smoothly bounded convex domain in $\R^n$. Then the 
reproducing kernel for $(H_{k}(\Omega),dV)$ satisfies 
\begin{equation}\label{kernelconvex}
B_k(x)=K_k(x,x) \simeq \min\left( \frac{k^n}{\sqrt{d(x)}}, k^{n+1}\right)\quad
\forall x\in\Omega.
\end{equation}
where  $d(x)$ denotes the distance of $x\in \Omega$ to the boundary of $\Omega$.
Thus it  satisfies the Bernstein-Markov property  \eqref{BM}  and it 
is of moderate growth \eqref{MG}.
Moreover, a necessary condition for the sequence $\Lambda_{k}$ of sets
of points on $\Omega$ to be sampling for $H_k(M)$ is that the density of 
sampling points is at least equal to the density of the equilibrium measure
$\mu_{eq}$ of $\Omega$, as $k\rightarrow\infty$, i.e. the pluripotential 
Nyquist bound \eqref{Nyquist} holds.
\end{thm}

\subsection{Interpolating sequences}
A natural companion problem to that of sampling sequences are the interpolating 
sequences. In the same  abstract point of view that we considered for sampling 
sequences we consider a sequence $H_{k}(M)$  of Hilbert spaces of functions on 
a set $M$ with reproducing kernels $K_{k}(x,y)$ and instead of a frames we 
consider Riesz sequences of normalized reproducing kernels (see 
Section~\ref{sec:interpolation} for the precise definitions).

Landau in \cite{la} studied also these sequences in the Paley-Wiener space, and 
his observation was that locally if a sequence  $\Lambda$ is 
interpolating then its density should be smaller than the local density of the 
space. We can again use the ideas inspired in \cite{no} to deal with the case of 
polynomials in real algebraic varieties.

Our main result to this problem is
\begin{thm}\label{generalinterp}
Let $M$ be an affine real algebraic variety equipped with the Lebesgue measure.
Then a necessary condition for a sequence $\Lambda_{k}$
of points on $M$ to be interpolating for the space $H_{k}(M)$ of polynomials of 
degree at most $k$, is that 
\begin{equation}\label{landauint}
\limsup_{k\rightarrow\infty}\frac{1}{N_{k}}\sum_{\lambda \in
\Lambda_{k}}\delta_{\lambda}\leq\mu_{eq}
\end{equation}
in the weak topology on $M$, where $\mu_{eq}$ denotes the normalized
equilibrium measure of $M$ and $N_k=\dim(H_k(M))$,  i.e. the following reversed 
pluripotential Nyqvist bound holds:
\[
\limsup_{k\rightarrow\infty}\frac{\#\{\Lambda_{k}\cap\Omega\}}{\#N_{k}}
\leq\frac
{\mu_{eq}(\Omega)}{\mu_{eq}(M)}
\]
for any given smooth domain $\Omega$ in $M$.
\end{thm}

\subsection{Discussion of the proof of Theorem~\ref{thm:sampling affine 
compact intro}}
Let us make some brief comments on the circle of ideas involved in the proof of
the previous theorems. First of all, since the sampling points uniquely
determine a polynomial $p_{k}$ on $M$ the total number $\#\Lambda_{k}$ of
sampling points at level $k$ is of course at least equal to the dimension
$N_{k}$ of $H_{k}(M)$. As a well-known guiding principle the necessary
conditions for sampling should come from a localized version of this argument
saying the asymptotic lower density of sampling points should at least be given
by the ``local dimension'' of the Hilbert space $(H_{k}(M),\left\Vert
\cdot\right\Vert _{L^{2}(d\mu_k)})$, were $\mu_k:=e^{-k\phi}d\mu$, i.e. by the
leading asymptotics of $N_{k}^{-1}$ times the function
\[
B_{k}(x):=\sum_{i=1}^{N_{k}}\left|p_{i}^{(k)}(x)\right|^{2} e^{-k\phi(x)}
\]
where $\{p_{i}^{(k)}(x)\}$ is any orthonormal base in the Hilbert
$(H_{k}(M),\left\Vert \cdot\right\Vert _{L^{2}(d\mu_k)})$ (note that integrating
$B_{k}(x)$ with respect to $dV$ indeed gives the dimension $N_{k}$ of
$H_{k}(M))$. The independence of the choice of base follows from the following
extremal representation of $B_{k}(x)$: 
\begin{equation}\label{eq:B-k as extremal quotient intro}
B_{k}(x):=\sup_{p_{k}\in H_{k}(M)}\frac{\left|p_{k}(x)\right|^{2}e^{-k\phi(x)}}
{\int_{M}\left|p_{k}\right|^{2}d\mu_k}
\end{equation}
From the point of view of general Hilbert space theory $B_{k}(x)$ may be
written as $B_{k}(x)=K_{k}(x,x)$ where $K_{k}(x,y)$ is the reproducing kernel of
the Hilbert space $(H_{k}(M),\left\Vert \cdot\right\Vert _{L^{2}(d\mu_k)})$,
i.e. the kernel of the orthogonal projection from $L^{2}(d\mu_k)$ to $H_{k}(M)$.
By the general results in \cite{bbw}
\begin{equation}\label{eq:asymptot of Bergman measure intro}
N_{k}^{-1}B_{k}(x)d\mu_k\rightarrow\mu_{eq}/\mu_{eq}(M)
\end{equation}
weakly on $M$ as $k\rightarrow\infty$, which in the view of the guiding
principle above thus gives a strong motivation for the previous theorems.
However, this guiding principle does not seem to hold in all generality and it
has to be complemented with some further asymptotic information of the full
reproducing kernel $K_{k}(x,y)$. This is already clear from Landau's classical
proof in the Paley-Wiener setting on $\R^{n}$ \cite{la}, where the decay
asymptotics of $K_{k}(x,y)$, away from the diagonal are needed in order to
construct functions $f_{k}$ which are well-localized on a given domain $\Omega$
(using suitable Toeplitz operators). Moreover, Landau's approach also relies on
certain submean inequalities for $f_{k}$ which pose difficulties in our general
setting. Instead we use a new approach to proving necessary conditions for
sampling, which is inspired by \cite{no} and \cite{lo}, where we reduce the
problem to establishing two asymptotic properties of $K_{k}(x,y)$ (given the
convergence of the Bergman function to the equilibrium measure):
\begin{itemize}
\item A growth property of $B_{k}$ 
\item A weak decay property of $\left|K_{k}(x,y)\right|$ away from the diagonal
\end{itemize}


The interpolation theorem~\ref{generalinterp} is proved in a similar way, but 
by replacing the moderate growth property with a Bernstein type inequality, see 
Theorem~\ref{Bernstein}.

One interesting feature is that although the statement of the problems studied 
are purely real, all the proofs rely on the process of complexification with 
one important exception: the off-diagonal estimate on the Bergman kernel 
(Theorem~\ref{prop:off-diagonal bergman kernel}) which exploits, in an essential 
way, the real structure.

\subsection{Further relations to previous results}

As explained above one important ingredient in the proof of 
Theorem~\ref{general} is the asymptotics for  $B_{k}(x)$ in formula 
\eqref{eq:asymptot of Bergman measure intro} established in \cite{bbw}, which 
holds generally under the Bernstein-Markov assumption on $(\mu,\phi)$. In turn, 
the latter result can be seen as a consequence of a very general  result in 
\cite{bbw} giving the convergence towards the equilibrium measure of 
$(\mu,\phi)$ for the normalized Dirac measure associated to a sequence of 
$N_{k}$ points under the condition that the points are asymptotic Fekete points 
for the weighted set $(\operatorname{Supp}(\mu),\phi)$. It may thus be tempting 
to try to deduce Theorem~\ref{general} in the present paper directly from the 
general convergence results in \cite{bbw} by the following tentative procedure: 
one removes points from a sampling sequence $\Lambda_{k}$ until one arrives at 
$N_{k}$ points, while keeping the sampling property. But as shown by a counter 
example in \cite[Example~2]{o-s} such a procedure is doomed to fail, 
already in 
the homogenouous case of the circle.

The point-wise asymptotics for $B_{k}(x)$ in Theorem~\ref{thm:sampling affine 
compact intro} (formula \eqref{eq:bounds on bergman function intro}) can be seen 
as an improvement - in the special case of a real algebraic manifold - of a 
classical result for regular compact subsets complex space going back to Siciak 
and Zaharyuta giving that $B_{k}(x)^{1/k}\rightarrow1$ point-wise on $M$ (this 
latter classical result was given a $\bar{\partial}-$proof by Demailly, 
\cite{demailly89} which can seen as a precursor to our proof). In view of the 
weak asymptotics \eqref{eq:asymptot of Bergman measure intro} and the bounds 
in formula \eqref{eq:inequalities for v_M in thm} below, it seems natural to 
conjecture that $B_{k}(x)/k^{n}$ in fact converges pointwise (in the almost 
everywhere sense) to the $L^{1}-$density of the equilibrium measure of $M$ (and 
similarly in the setting of a convex domain; as in the one dimensional setting, 
see \cite{totik}). This would be a real analog of the point-wise asymptotics for 
$B_{k}(x)$ in \cite{berman}, where the role of $M$ is played by a complex 
projective variety endowed with a hermitian holomorphic line bundle (the case of 
positive curvature is a fundamental result in complex geometry, due to Bouche 
\cite{bouche} and Tian \cite{tian}).

On the other hand, in the line bundle setting, the asymptotics for 
$\left|K_{k}(x,y)\right|^{2}$ in Theorem~\ref{prop:off-diagonal bergman kernel} 
are only known in the case of a line bundle with positive curvature, but it 
seems natural to expect that they hold in general (see \cite{berman} for some 
results in this direction in connection to the study of fluctuations of linear 
statistics of determinantal point processes).

Finally, we recall that in the one-dimensional case there is a vast
litterature on various asymptotic results for orthogonal polynomials,
in particular in connection to random matrix theory. For example,
the asymptotics of the scaled reproducing kernel 
$k^{-n}K_{k}(x+\frac{a}{k},x+\frac{b}{k})$
have been established, in connection to the question of universality, under 
very general condition on a given measure $\mu$ on the real
line, when $x$ is a fixed point in the ``bulk'' of the support
of $\mu$; see for example \cite{lubinsky} (and similar scaling result holds at 
the ``edge'' of the support). However, there seem to be very few results 
in the higher dimensional setting (but see \cite{xu}
and \cite{k-l}, where the case of the ball and the simplex in is
settled). It would be interesting to extend the asymptotics in the
present paper to study similar universality questions in higher dimensions
and we leave this as a challenging problem for the future.
\subsection*{Acknowldegements:}
The authors are grateful to Ahmed Zeriahi for very useful discussions
regarding the Bernstein inequality.

\section{Pluripotential theory and asymptotics of real orthogonal polynomials }

\subsection{Setup}\label{setup}

Let $M$ be an $n$-dimensional affine real algebraic variety, which
is non-singular and compact. In particular, $M$ is the common zero-locus
of a collection of real polynomials $p_{1},...,p_{r}$ in $\R^{m}$.
We denote by $H_{k}(M)$ the real vector space consisting of the functions
on $M$ which are restrictions of real polynomials in $\R^{m}$ of
total degree at most $k$. We will also consider the ``complexifications''
$X$ and $H_{k}(X)$ of the real variety $M$ and the real vector
space $H_{k}(M)$, respectively. More precisely, $X$ is the complex
algebraic variety in $\C^{m}$ defined by the common complex zeros of the ideal 
defining $M$
and $H_{k}(X)$ is the complex vector space consisting of restrictions
to $X$ of polynomials in $\C^{m}$ of total degree at most $k$.
Then $M$ is indeed the real part of $X$ in the sense that it consists
of all points in $z$ in $X$ such that $\bar{z}=z$ and real vector
space $H_{k}(M)$ is the the real part of $H_{k}(X)$ in the sense
that it consists of all $p_{k}$ in $H_{k}(X)$ such that $\overline{p_{k}}=p_{k}$
(restricted to $M$). Denoting by $\bar{X}$ the closure of $X$ in
$\P^{m}$, which defines a compact (possibly singular) complex projective
algebraic variety we will, in the usual way, identify $H_{k}(X)$
with the space $H^{0}(X,\mathcal{O}_{X}(1)^{\otimes k})$ of all global
holomorphic sections of the line bundle $\mathcal{O}_{X}(1)^{\otimes 
k}\rightarrow X$.

We will denote by $K_k(x,y)$ the Bergman reproducing kernel of $H_k$ equipped 
with the $L^2$-norm induced by a given weighed measure $(\mu,\phi)$
and moreover we will use the notation $B_k(x)=K_k(x,x)e^{-k\phi(x)}$.

\begin{example}
Let $M$ be the unit-circle realized as the zero-set in $\R^{2}$
of $p(x,y)=x^{2}+y^{2}-1$. Setting $x=\cos\theta$ for $\theta\in[0,2\pi]$
we may identify $H_{k}(M)$ with the space $H_{k}([0,2\pi])$ of all
Fourier series on $[0,2\pi]$ ``band-limited'' to $[-k,k]$ i.e.
spanned by $1$ and $\cos m\theta$ and $\sin m\theta$ for 
$m\in[\text{1},k]\cap\Z$.
More precisely, $H_{k}([0,2\pi])$ is the pull-back of $H_{k}(M)$
under the corresponding map from $[0,2\pi]$ to $M$. To see the relation
to the more standard setting where $M$ corresponds to the unit-circle
$S^{1}$ in $\C$ with complex coordinate $\tau$ (equal to $e^{i\theta}$
on $S^{1})$ we note that the embedding $F$ of $\C^{*}$ in $\C^{2}$
given by 
\[
z:=(\tau+\tau^{-1})/2\quad  w=(\tau-\tau^{-1})/2i
\]
maps $\C^{*}$ to the complex quadric $X$ cut out by $p(z,w):=z^{2}+w^{2}-1$
and the unit-circle $S^{1}$ in $\C^{*}$ is mapped to the real part
$M$ of $X$. Indeed, $\tau\in S^{1}$ if and only if $\bar{\tau}=\tau^{-1}$
iff $z(\tau)=\Re\tau(=\cos\theta)$ and $w(\tau)=\Im\tau(=\sin\theta)$.
The pull-back of $H_{k}(X)$ under $F$ is the space of Laurent polynomials
on $\C$ spanned by the monomials $\tau^{m}$ for $m\in[-k,k]\cap\Z$.
Hence, the real part of $F^{*}H_{k}(M)$ is indeed spanned by $1$
and $\cos m\theta(=\Re\tau^{m})$ and $\sin m\theta(=\Im\tau^{m})$
for $m\in[\text{1},k]\cap\Z$. 
\end{example}

\begin{definition}[Bernstein-Markov]\label{defBM}
The standard assumption on the pair ($\mu$, $\phi$) is that it satisfies the
Bernstein-Markov property (with respect to the support of $\mu$), which may be
formulated as the property that the reproducing kernel $K_{k}(x,y)$ have
sub-exponential growth on the diagonal, i.e. for any $\epsilon>0$ there exists a
positive constant $C_\epsilon$ such that 
\begin{equation}\label{BM}
B_{k}(x)\leq C_\epsilon e^{\epsilon k}
\end{equation}
uniformly on the support of $\mu$. 
\end{definition}
For our general results to hold we will also need the following technical
regularity of the growth assumption on the reproducing kernel along the
diagonal:
\begin{definition}[Moderate growth]\label{defMG}
We say that $H_k$ has a reproducing kernel with moderate growth if
\begin{equation}\label{MG}
 K_{k+1}(x,x)\leq C K_{k}(x,x)
\end{equation}
on the support of $\mu$. More precisely, for our purposes the constant $C$ may 
be
replaced by any sequence with growth of the order $o(k)$. 
\end{definition}
Anyway, in all examples that we are aware of the constant $C$ in \eqref{MG} may
actually be replaced by a sequence tending to one as $k\rightarrow\infty$
(which, by iteration, actually implies the Bernstein-Markov property
\eqref{BM}). All the measures $\mu$ supported in $M$ that we will consider are
\emph{non-degenerate}, in the sense that $\|f\|\ne 0$ if $0\ne f\in H_k(M)$. 
Otherwise
if the support of $\mu$ is contained in the zeros of a non vanishing polynomial
in $H_k(M)$ one may replace $M$ with a subvariety. 

While Definition~\ref{defBM} is standard (see \cite{bbw} and references 
therein), Definition~\ref{defMG} appears to be new. We expect it to hold in 
great generality and we will establish it in the situations relevant to the 
present paper.

\subsection{The extremal function attached to a real affine 
variety}\label{extremal}

Recall that the \emph{Lelong class} $\mathcal{L}(\C^{m})$ is the
convex space of all plurisubharmonic (psh, for short) functions $\phi$
on $\C^{m}$ with logarithmic growth, in the sense that 
$\phi\leq\log(1+|z|^{2})+C$.
The restriction of this space to $X$ will be denoted by $\mathcal{L}(X)$
and it may be identified with the space of all (singular) metrics
on the line bundle $\mathcal{O}(1)_{\bar{X}}\rightarrow\bar{X}$ with
positive curvature current, see \cite{bbw} and references therein for further 
background. The \emph{Siciak extremal function} (sometimes
called the equilibrium potential) of a compact and non-pluripolar
subset $K$ of $X$ and a weight $\psi\in \C(K)$ is the function in 
$\mathcal{L}(X)$ defined as
the upper semi-continuous regularization $v_{K}^{*}$ of the envelope
\[
v_{K,\psi}(x)=\sup_{\phi\in\mathcal{L}(X)}\left\{ \phi(x):\quad\phi\leq 
\psi\text{ on }K\right\} 
\]
and the weighted set $(K,\psi)$ is called \emph{regular} if 
$v_{K,\psi}^{*}=v_{K,\psi}$. The 
Monge-Amp\`ere
measure 
\[
\mu_{K,\psi}:=MA(v_{K,\psi})
\]
is called the\emph{ (pluripotential) equilibrium measure} of $(K,\psi)$
and it is supported on $K$ (when there is no risk of confusion we
will write $\mu_{K,\psi}:=\mu_{eq})$ 

In the following we will take $K:=M$ as above, which is thus embedded
in the complex affine variety $X$ in such a way that $M=X\cap\{y=0\}$,
where $y$ denotes the imaginary part of $z\in\C^{n}$. We will also take 
$\psi=0$. 
\begin{prop}
\label{thm:equil weight for M affine real smooth}Let $M$ be an $n$-dimensional
affine real algebraic variety, which is non-singular and compact and
denote by $v_{M}$ its extremal function, defined on the complexification
$X$ of $M$. Then $M$ is non-pluripolar and regular. Moreover, there
exists a constant $C$ such that 
\begin{equation}\label{eq:inequalities for v_M in thm}
\frac{1}{C}\left|y\right|\leq v_{M}\leq C\left|y\right|
\end{equation}
 in a neighborhood of $M$ in $X$. In particular, the equilibrium
measure $\mu_{M}$ is absolutely continuous with respect to the Lebesgue
measure $dV_{M}$ on $M$ and its density is bounded from above and
below by positive constants: 
\begin{equation}\label{eq:ineq for mu on M in thm}
\frac{1}{D}dV_{M}\leq\mu_{M}\leq DdV_{M}
\end{equation}
 on $M$. 
\end{prop}
\begin{proof}
The lower bound on $v_{K}$ follows from a simple max construction.
Indeed, we may after a scaling assume, as before, that $|z|<1$ on $M$ and then
set 
\[
\psi:=\begin{cases}\max\{|y|/C,\log|z|^{2}\}& \text{ if } |z|\le 2\\
       \log |z|^2 &\text{ if } |z|>2.
      \end{cases}
\]
with $C$ sufficiently large to ensure that $\psi$ is continuous
along $|z|=2$. Since, $\psi=|y|/C$
close to $M$ the function $\psi$ is a contender for the sup defining
$v_{M}$ and hence $\psi\leq v_{M}$, which proves the lower bound
in \eqref{eq:inequalities for v_M in thm}. The proof of the upper bound
is more involved: We know that $v_M\in L^\infty_{\text{loc}}(X)$ because $M$ is
algebraic, see \cite{sadullaev}. Moreover there is a distance $d$ such that
$g(r)=\sup_{z\in X: d(z,M)=r} v_M(z)$ is convex in $r$ in $[0,\delta]$, see the
proof of Theorem~\ref{tubular} for details. Therefore $g(r)\le C r$ and the
upper bound in \eqref{eq:inequalities for v_M in thm} follows.

Anyway, for the proof of the lower bound
in \eqref{eq:bounds on bergman function intro} we will only
need the lower bound in \eqref{eq:inequalities for v_M in thm}. Note
also that combining \eqref{eq:bounds on bergman function intro}
and the asymptotics \eqref{eq:asymptot of Bergman measure intro} immediately
gives the inequalities \eqref{eq:ineq for mu on M in thm}
(which also follow from the inequalities \eqref{eq:inequalities for v_M in thm}
by the comparison principle for the Monge-Ampere measure, see 
\cite[Lemma~2.1]{BedTay}.
\end{proof}

\subsection{The proof of the lower bound on \texorpdfstring{$B_{k}$}{Bk} in 
\texorpdfstring{\eqref{eq:bounds on bergman function intro}}{(\ref{eq:bounds on 
bergman function intro})}}

Denote by $B_{kv_{M}}$ the Bergman function on $X$ defined by the
$L^{2}$-norm on $H_{k}(X)$ induced by the
measure $e^{-kv_M}dV_{X}$. The idea of the proof is to first show that 
\begin{equation}\label{first step}\tag{i}
B_{kv_{M}}\geq Ck^{2n}\quad \text{on }M,
\end{equation}
and then that 
\begin{equation}\label{second step}\tag{ii}
B_{k}\geq Ck^{-n}B_{kv_{M}}\quad \text{on }M.
\end{equation}
This would clearly imply the result in question. However, for technical
reasons we will only show a slightly weaker version of these inequalities
(needed for \eqref{first step}) where $v_{M}$ is replaced by 
\[
v_{M}^{\epsilon}:=v_{M}(1-\epsilon)+\epsilon\psi
\]
 where $\psi$ is a continuous function in $\mathcal{L}(X)$ such
that $\psi=v_{M}^{2}/C$ in a neighborhood of $M$. Here $\epsilon$
is a sufficiently small positive number which is fixed once and for
all. To see that such a function $\psi$ exists we may after scaling
assume that $|z|<1$ on $M$ and then simply set $\psi:=\max\{v_{M}^{2}/C, 
\log|z|^{2}\}$ when $|z|<2$ and $\phi:=\log|z|^2$ when $|z|>2$. The constant
 $C$ is taken sufficiently large to ensure that $\psi$ is continuous at $|z|=2$.

Let us start with the proof of \eqref{first step}. To simplify the notation we
will assume that $n=1$ (but the general proof is essentially the
same). To this end fix a point in $M$ and introduce local holomorphic
coordinates $z$ on $U$ in $X$, centered a the fixed point, such
that $M=\{y=0\}$ locally, i.e. on $U$ (not to be confused with the
global coordinates on $\C^{m}$ and $\R^{n}$, respectively). The
idea is to first construct a local function $f_{k}$, holomorphic
on $U$ such that 
\begin{equation}\label{eq:local model bound}
\frac{\left|f_{k}(0)\right|^{2}}{\int_{U}\left|f_{k}\right|^{2}e^{-kv_{M}}dV}
\geq k^{2}/C
\end{equation}
and then perturb $f_{k}$ slightly to become a polynomial $p_{k}$
by solving a global $\bar{\partial}$-equation on $\bar{X}$ with
an $L^{2}$-estimate. 

There is no loss of generality assuming that $f_k(0)=1$. Working in a 
local coordinates and reescaling \eqref{eq:local model bound} it is enough to 
prove that there is a function $f\in \mathcal H(\C)$ such that $f(0)=1$ and
\[
 \int_{\mathbb C} |f(z)|^2 e^{-C|\Im z|} < \infty,
\]
then $f_k(z)=f(kz)$ will satisfy \eqref{eq:local model bound}. The 
function $f(z)= \sinc^{2}(C z/2)$ has the desired properties.

\subsubsection{Modification and globalization}

Let now $\chi$ be a smooth cut-off function supported on $U$ (say
equal to one on $U/2)$. In view of standard globalization arguments
the problem with the present setting is that $\bar{\partial}$ of
the global function $\chi f_{k}$ on $X$ does not have a small weighted
$L^{2}$- norm (compared to the weighted norm of $f_{k})$. The reason
is that $e^{-k2|y|}$ is only well localized in the $y$-direction.
To bypass this difficulty we will instead replace $v_{M}$ with 
$v_{M}^{\epsilon}$
and modify $f_{k}$ accordingly as follows. First observe that, by
definition, 
\[
\int_{U}\left|g_{k}\right|^{2}e^{-kv_{M}^{\epsilon}}dV\leq\int_{U} 
\left|g_{k}\right|^{2}e^{-kv_{M}(1-\epsilon)}e^{-k\epsilon4|y|^{2}}dV
\]
for any $g_{k}$. We next observe that 
$\left|e^{z^{2}}\right|^{2}e^{-4|y|^{2}}=e^{-2|z|^{2}}$
and hence setting $g_{k}:=f_{k}e^{k\epsilon z^{2}}$ gives 
\[
\left|g_{k}\right|^{2}e^{-k4\epsilon|y|^{2}}=\left|f_{k}\right|^{2}e^{
-2k\epsilon|z|^{2}}\leq\left|f_{k}\right|^{2}
\]
 and $g_{k}(0)=f_{k}(0)$. In particular, 
\[
\frac{\left|g_{k}(0)\right|^{2}}{\int_{U}\left|g_{k}\right|^{2}e^{-kv_{M}^{
\epsilon}}dV}\geq\frac{\left|f_{k}(0)\right|^{2}}{\int_{U}\left|f_{k}\right|^{2}
e^{-k(1-\epsilon)v_{M}}dV}\geq k^{2}/C_{\epsilon}
\]
Here the optimal constant $C_{\epsilon}$ is slightly smaller than
the previous optimal $C$, but on the other hand we have gained a
Gaussian factor that we will next exploit. The point is that 
$\bar{\partial}(\chi g_{k})=\bar{\partial}\chi g_{k}$
is supported where $|z|>1/4$ and hence 
\[
\begin{split}
\int_{U}\left|\bar{\partial}(\chi g_{k})\right|^{2}e^{-kv_{M}^{\epsilon}}dV\leq 
C\int_{1/4\leq|z|\leq2}\left|g_{k}\right|^{2}e^{-kv_{M}(1-\epsilon)}e^{
-k\epsilon4|y|^{2}}dV=\\
C\int_{1/4\leq|z|\leq2}\left|f_{k}\right|^{2}e^{-kv_{M}
(1-\epsilon)}e^{-k\epsilon2|z|^{2}}dV
\end{split}
\]
Estimating the Gaussian factor $e^{-k\epsilon2|z|^{2}}$ with its
sup, i.e with $e^{-k2\epsilon/4^{2}}$ thus gives the bound 
\[
\int_{U}\left|\bar{\partial}(\chi g_{k})\right|^{2}e^{-kv_{M}^{\epsilon}}dV\leq 
O(e^{-\delta k})\int_{U}\left|f_{k}\right|^{2}e^{-kv_{M}(1-\epsilon)}
\]
Here and henceforth $O(e^{-\delta k})$ denotes a term which is exponentially
small in $k$ (recall that that $\epsilon$ is a small number which
is fixed once and for all). 

With this local estimate in place we can now apply a standard globalization
argument: using $L^{2}$-estimates for $\bar{\partial}$ on the line
bundle $\mathcal{O}(1)_{X}$ over $\bar{X}$, or more precisely (if
the latter variety is singular) on its pull-back to a smooth resolution
of $\bar{X})$ there exists a smooth function $u_{k}$ on such that
$p_{k}:=g_{k}-u_{k}$ is in $H_{k}(X)$ and 
\[
\bar{\partial}u_{k}=\bar{\partial}(\chi 
g_{k}),\quad \int_{X}\left|u_{k}\right|^{2}e^{-kv_{M}^{\epsilon}}dV\leq 
C\int_{U}\left|\bar{\partial}(\chi g_{k})\right|^{2}e^{-kv_{M}^{\epsilon}}dV
\]
(strictly speaking to apply $L^{2}$-estimates we have to slightly
modify the weight $v_{M}^{\epsilon}$ with a
$k$-independent term to ensure that the corresponding metric on the
line bundle $k\mathcal{O}(1)_{X}$ has a sufficiently large uniform
lower bound on its curvature form, but this only changes the $L^{2}$-estimates
with an overall multiplicative constant, which is harmless).  This is a standard 
procedure; for a precise statement which also applies in the singular setting 
see, for example \cite[Section~2]{berman2}.

By the previous estimate this means that 
\[
\int_{X}\left|u_{k}\right|^{2}e^{-kv_{M}^{\epsilon}}dV\leq O(e^{-\delta 
k})\int_{U}\left|f_{k}\right|^{2}e^{-kv_{M}(1-\epsilon)}
\]
Moreover, applying the mean value property for holomorphic functions
on a small coordinate ball then gives 
\[
u_{k}(0)\leq Ck^{2}\int_{X}\left|u_{k}\right|^{2}e^{-kv_{M}^{\epsilon}}dV\leq 
O(e^{-\delta k})\int_{U}\left|f_{k}\right|^{2}e^{-kv_{M}(1-\epsilon)}
\]
Hence, 
\[
\begin{split}
\frac{\left|p_{k}(0)\right|^{2}}{\int_{X}\left|p_{k}\right|^{2}e^{-kv_{M}^{
\epsilon}}dV}=\frac{\left|g_{k}(0)-u_{k}(0)\right|^{2}}{\int_{X}\left|\chi 
g_{k}-u_{k}\right|^{2}e^{-kv_{M}^{\epsilon}}dV}\geq\\
\frac{\left|g_{k}(0)\right|^{
2}-O(e^{-\delta 
k})\int_{U}\left|f_{k}\right|^{2}e^{-kv_{M}(1-\epsilon)}}{\int_{U}\left|\chi 
g_{k}\right|^{2}e^{-kv_{M}^{\epsilon}}dV+O(e^{-\delta 
k})\int_{U}\left|f_{k}\right|^{2}e^{-kv_{M}(1-\epsilon)}}
\end{split}
\]
But $\left|g_{k}(0)\right|^{2}=\left|f_{k}(0)\right|^{2}$ and 
$\left|g_{k}\right|^2e^{-kv_{M}^{\epsilon}}\leq\left|f_{k}\right|^{2}e^{-kv_{M}
(1-\epsilon)}$.
Moreover, as explained above 
$\int_{U}\left|f_{k}\right|^{2}e^{-kv_{M}|(1-\epsilon)}=O(k^{2})$
and hence we get just as above 
\[
\frac{\left|p_{k}(0)\right|^{2}}{\int_{X}\left|p_{k}\right|^{2}e^{-kv_{M}^{
\epsilon}}dV}\geq C_{\epsilon}k^{2}(1+O(e^{-\delta k}))
\]
 which concludes the proof of the bound \eqref{first step}
\[
B_{kv_{M}^{\epsilon}}\geq Ck^{2n}\quad \mbox{on } M.
\]

\subsubsection{The inequality between $B_{k}$ and $B_{kv_{M}}$}

First observe that it is enough to prove the following lemma where
now $y$ denotes the imaginary part of $z\in\C^{n}$ (so that $X\cap\{y=0\}=M\}):$
\begin{lem}
\label{lem:comparison of l2 norms for polyn}Let $U_{k}$ be the set
of all points in $X$ such that $|y|\leq1/k$ (which defines a neighborhood
of $M$ in $X)$. Then there exists a constant $C$ such that 
\[
\int_{M}|p_{k}|^{2}dV\leq C\frac{1}{Vol(U_{k})}\int_{U_{k}}|p_{k}|^{2}dV
\]
 for any polynomial of total degree at most $k$.
\end{lem}
Indeed, since the function $v_{M}$ on $X$ is comparable to $|y|$
close to $M$ (by Theorem~\ref{thm:equil weight for M affine real smooth})
and in particular $kv_M^\epsilon$ is uniformly bounded on $U_{k}$, we then get 
that
\[
\int_{M}|p_{k}|^{2}dV\leq 
C'\frac{1}{k^{n}}\int_{X}|p_{k}|^{2}e^{-kv_{M}^{\epsilon}}dV
\]
It follows immediately that 
\[
B_{k}\geq k^{n}/C'B_{kv_{M}^{\epsilon}}
\]
on $M$, which combined with the inequality \eqref{first step} thus concludes the
proof of the lower bound in \eqref{eq:bounds on bergman function intro},
given Lemma~\ref{lem:comparison of l2 norms for polyn}, to whose
proof we next turn.

For any $x\in M$ there are constants $C$ and $r_0$ such that for any $r<r_0$,
\[
 |f(x)|^2 \le \frac C{r^{2n}}\int_{X\cap B(x,r)} |f(y)|^2\, dV_X(y). 
\]
for any $f$ holomorphic in $X$. In particular if we integrate over $x\in X$ a
polynomial of degree $k$ taking $r=1/k$ we get
\[
\begin{split}
 \int_M |p_k|^2dV_M \le Ck^{2n} \int_{U_k} |p_k(y)|^2 V_M(B(y,1/k)\cap M)\,
dV_X(y)\\
\le C k^n\int_{U_k} |p_k|^2dV_X\le \frac{C}{Vol(U_k)}\int_{U_k}|p_k|^2dV_X.\qed
\end{split}
\]

\subsection{The \texorpdfstring{$L^q$}{Lq} Bernstein inequality}

Let $M$ be a smooth compact algebraic variety in $\R^m$ of dimension $n$. 

Given a polynomial $p\in \mathcal P_k(\mathbb R^m)$  and $x\in M$ we denote by
$\nabla_t p(x)$ the tangential gradient of $p$ along the manifold $M$. The
following Bernstein type inequality holds:
\begin{thm}\label{Bernstein}
Let $q\in [1,\infty]$, then there is a constant $C_q$ such that
\[
 \|\nabla_t p\|_{L^q(M)} \le C_q \deg(p) \|p\|_{L^q(M)}. 
\]
\end{thm}
The case $q=\infty$ was proved in \cite{BosLevMilTay}. We prove now the case
$q=1$ and the others follow by interpolation.

Let
$X$ be a complexification of $M$, i.e, an algebraic variety in $\C^m$ such that
$M=X\cap \R^m$. 
We denote by $U_r\subset X$ the neighborhood of $M$ defined as $U_r=\{x\in
X: d(x,M) < r\}$. By the Cauchy inequalities we have that for any $x\in M$ and
any $f\in \mathcal H(X)$:
\[
 |\nabla_t f(x)|\lesssim \frac 1{r^{2n+1}} \int_{B(x,r)} |f(y)|dV_X(y).
\]
and integrating over $M$ we have
\[
 \int_M |\nabla_t f(x)|dV_M \lesssim \frac 1{r^{2n+1}} \int_M
\int_{B(x,r)}\!\!\!
|f(y)|dV_X(y)dV_M(x)\lesssim \frac 1{r^{n+1}}\int_{U_r} |f(y)|dV_X(y).
\]
Therefore Theorem~\ref{Bernstein} follows from the following
result:
\begin{thm}\label{tubular}
 There is $C>0$ such that for all polynomials $p_k$ of degree $k$, the
following inequality holds:
\[
 \int_{U_{1/k}} |p_k|dV_X \le C k^{-n} \int_{M}|p_k|dV_M.
\]
\end{thm}
\begin{proof}
In order to estimate the integral over $U_{1/k}$ we will integrate along
surfaces surrounding $M$. These surfaces will be level sets of plurisubharmonic
functions with Monge-Ampere 0. In this setting there is a generalization of
Hadamard three circles theorem due to Demailly that will be used, see
\cite{demailly}. In order to use this technique we need that the psh-function
that defines the level sets is smooth out of $M$, and its square must be
smooth. We can use the function provided by Guillemin and Stenzel in
\cite{guillemin} in their study of Grauert tubular neighbourhoods around real
analytic manifolds. We present the setting:

 Take $\psi$ a plurisubharmonic function in a neighborhood $U$ of $M$ in $X$
defined as 
\[
 \psi(z)=d(z,M),
\]
where the distance $d$ is given by a metric provided in a Grauert tubular
neighborhood $U$ as in \cite{guillemin}. The function $\psi$ satisfies
$(dd^c\psi)^n=0$ in $U\setminus M$, $\psi^2$ is a real analytic K\"ahler
potential in $U$ and $(dd^c(\psi^2))^n$ is comparable to the volume form in $X$
in a neighborhood of $M$.

We use the same notation as in \cite{demailly}. Consider the pseudospheres
$S(r)=\{z\in U;\psi(z)=r$, $r>0$ and the positive measures $\mu_r$ supported on
$S_r$ that are defined as 
\[
 \mu_r(h):= \int_{S(r)} h (dd^c\psi)^{n-1}\wedge d^c\psi
\]
for any $h\in C(U)$. When $r=0$, then we define
\[
 \bar\mu_0(h):=\int_{M} h (dd^c \psi)^n.
\]
We have that $\mu_r(h)$ is continuous for $r> 0$ with $\mu_r(h)\to\bar\mu_0(h)$
as $r\to 0^+$ see \cite[Theorem~3.2]{demailly}.

Moreover $(dd^c \psi)^n$ which is supported on $M$ is comparable to the volume 
form in $M$. This is so because at any point $z\in M$ we can write local 
holomorphic coordinates such that $M$ corresponds to ${z\in \mathbb C^n: \Im 
z=0}$. In this coordinates $\psi(z)$ is comparable to $|\Im z|$ and since $(dd^c 
|\Im z|)^n$ is the Lebesgue measure on $\mathbb R^n$ then by the comparison 
principle for the Monge-Ampere measure, see \cite[Lemma~2.1]{BedTay}, the 
measure $\bar\mu_0$ is locally comparable to the volume form and $M$ being 
compact it is globally comparable.

Take the psh function $V=\log|p_k|$, then 
\cite[Corollary~6.6(a)]{demailly} says that the function 
\[
 u(r)=\log \mu_r(e^V), \quad r>0, \quad u(0)=\log \bar\mu_0(e^V)
\]
is convex and increasing in $r$. We fix $R>0$ such that $S_R$ belongs to the
tubular neighborhood $U$. The convexity of $u$ implies that for any $r>0$
\begin{equation}\label{convex}
 u(r)\le u(0) \frac{R-r}{R} + u(R)\frac {r}R.
\end{equation}
We have that
\[
 u(0)=\log \int_{M} |p_k|\, d\bar\mu_0.
\]
We are going fix $R$ such and estimate $u(R)$. Since $p_k$ is a
polynomial of degree $k$ we have that by the Bernstein-Walsh estimate
\[
 \sup_{S_R} |p_k| \le  \sup_{S_R} e^{k\phi_M(z)} \sup_{M}|p|,
\]
where $\phi_M$ is the Siciak extremal function defined as
\[
\phi_M(z)=\max\Bigl\{0,\sup \{\frac 1{\deg(p)}\log |p(z)|:\ p\in \mathcal
P_k(\mathbb C^m), \deg(p)>0, \sup_M |p|\le 1\}\Bigr\}
\]

It is a well-known theorem of Sadullaev, see \cite{sadullaev} that if $X$ is
algebraic then $\phi_K\in L^\infty_{loc}(X)$ for any non-pluripolar compact set
$K$ relative to $X$. Certainly $M$ is non-pluripolar relative to $X$ since it is
totally real. Therefore:
\[
 \sup_{S_R} |p_k| \le  C^k \sup_{M}|p|
\]
Moreover, we need the following Bernstein-Markov type inequality:
\begin{equation}\label{bmlight}
 \sup_M |p| \le C_M^k \int_M |p|.
\end{equation}
This is easier than the standard Bernstein-Markov property since we are not
requiring that $C_M$ is close to 1. In our case \eqref{bmlight} is a
special case of \cite[Theorem~4.1]{brudny}.

Finally $\sup_{S_R} |p_k|\le C^k  \int_M |p|$, and we have that 
$u(R)\le (C+\log\|p\|_{L^1(M)})k$.

Therefore if $r=1/k$ and using the convexity \eqref{convex}  we deduce that
\[
 u(1/k)\le u(0) + C,
\]
Since $u(r)$ is increasing we have that  for any $r<1/k$
\[
\int_{S_r} |p_k| d\mu_r \le C  \int_{M} |p_k| d\bar\mu_0.
\]

But the measures $\mu_r$ desintegrate the form $(dd^c\psi)^{n-1}\wedge
d\psi\wedge d^c\psi$, see \cite[Proposition~3.9]{demailly}, and we 
have that
\[
 \int_{0}^{1/k}r^{n-1}\int_{S(r)} |p_k| d\mu_r =
 \int_{\psi<{1/k}} |p_k(z)| \psi(z)^{n-1}  (dd^c\psi)^{n-1}\wedge d\psi\wedge
d^c\psi 
\]
Moreover 
\[
\psi(z)^{n-1}  (dd^c\psi)^{n-1}\wedge d\psi\wedge d^c\psi=(dd^c(\psi^2))^n
\]
But $\psi^2$ is a real analytic K\"ahler potential in $X$, see \cite{guillemin}.
Thus $(dd^c(\psi^2))^n$ is equivalent to the original volume form $V_X$ in a
neighborhood of $M$ in $X$:

\begin{equation*}
 \int_{\psi<{1/k}} |p_k|dV_X \simeq \int_{0}^{1/k}r^{n-1}\int_{S(r)} |p_k|
d\mu_r
\end{equation*}
\end{proof}
\begin{remark}
With the same proof, for any $1\le q <\infty$,
\[
 \int_{U_{1/k}} |p_k|^q dV_X \lesssim k^{-n} \int_{M}|p_k|^q dV_M.
\]
It is also true that
\[
 \sup_{U_{1/k}}|p_k| \lesssim \sup_{M} |p_k|.
\]
The proof is the same, but instead of 
\cite[Corollary~6.6(a)]{demailly} one uses 
that
\[
 u(r)=\sup_{S_r} \log |p_k|,
\]
is a convex function of $r$, see 
\cite[Corollary~6.6(b)]{demailly}.
\end{remark}
\begin{remark}
For any $x\in M$ we consider a 
ball $B_X(x,1/k)$ in the complexified manifold $X$. By the submean value 
property we have that	
\[
 |p_k(x)|^2 \lesssim k^{2n} \int_{B_X(x,1/k)}|p_k(y)|^2\, dV_X(y) \lesssim
k^{n}\int_{M}|p_k|^2\, dV_M.
\]
Therefore $K_k(x,x)\lesssim k^n$ and we have proved the upper inequality in 
\eqref{eq:bounds on bergman function intro}. We include the argument for 
completeness but this upper bound is well known and it follows from the 
arguments in \cite{zeriahi}.
\end{remark}

It is also possible to prove a converse result to Theorem~\ref{Bernstein}:
\begin{thm}\label{converse}
 Let $M$ be a smooth compact submanifold in $\mathbb R^m$. If there is a
constant $C>0$ such that for some $q\in[1,\infty]$,
 \[
 \|\nabla_t p\|_{L^q(M)} \le C\deg(p) \|p\|_{L^q(M)}, 
\]
for all polynomials $p\in \mathcal P(\mathbb R^m)$, then $M$ is algebraic.
\end{thm}
\begin{proof}
We will need a definition
\begin{definition} A sequence of finite sets
$\{\Lambda_k\}$ is an $\epsilon$-net if it is uniformly separated and 
$1\le\sum_{\lambda\in \Lambda_k}\chi_{B(\lambda,\varepsilon/k)}(x)\le C$, for
all $x\in M$ and $k>0$
\end{definition}
 By an application of the Vitali covering lemma it
is possible to construct $\epsilon$-nets for arbitrarily small $\varepsilon$
where the constant $C=C_M$ depends on the dimension of $M$ but
not on $\varepsilon$.

Given an  $\epsilon$-net $\Lambda=\Lambda(\varepsilon)$ we
denote by $l_k=\#\Lambda_k$. We may
define:
$T_k: \mathcal P_k(M)\to \mathbb R^{l_k}$ as 
\[
 T_k(p)(\lambda)= p_{B_M(\lambda,\varepsilon/k)}:=
\fint_{B_M(\lambda,\varepsilon/k)} p\, dV_M \qquad \forall \lambda\in \Lambda_k.
\]
We will prove now that if $\varepsilon$ is small enough then
\begin{equation}\label{pseudosamp}
 \int_M |p_k|^q \lesssim \frac 1{k^n} \sum_{\lambda\in \Lambda_k}
|T_k(p)(\lambda)|^q.
\end{equation}
If this is the case, then $T_k$ is one to one and $\dim (\mathcal P_k(M))\le
l_k\simeq k^n$ where $n=\dim(M)$ which is much smaller than $k^m$. Thus $M$ is
algebraic.

Let us prove \eqref{pseudosamp}. 
\[
  \int_M |p_k|^q\, dV_M\lesssim  \frac 1{k^n} \sum_{\lambda\in
\Lambda_k}
|T_k(p)(\lambda)|^q+
\sum_{\Lambda_k}\int_{B_M(\lambda,\varepsilon/k)}
|p(x) -p_{B_M(\lambda,\varepsilon/k)}|^q\, dV_M.
\]
By the Poincaré inequality:
\[
 \int_{B_M(\lambda,\varepsilon/k)}
|p(x)-p_{B_M(\lambda,\varepsilon/k)}|^q\, dV_M\lesssim
 \frac{\varepsilon^q}{k^q} \int_{B_M(\lambda,\varepsilon/k)} |\nabla_t p|^q\,
dV_M.
\]
By Theorem~\ref{Bernstein}
\[
 \int_M |p_k|^q\, dV_M \lesssim \frac 1{k^n} \sum_{\lambda\in \Lambda_k}
|T_k(p)(\lambda)|^q + \varepsilon  \int_M |p_k|^q\, dV_M.
\]
and if $\varepsilon$ is small enough then \eqref{pseudosamp} follows.

\end{proof}

\subsection{Applications}
We will use now the Bernstein inequality to get some more information on
sampling and interpolation sequences of finite sets.

\begin{definition}
 A sequence of measures $\{\mu_k\}_k$ is said to be a uniformly sequence of
Carleson
measures if there is a $C>0$ such that
\begin{equation}\label{Carleson}
 \int_{M}|p_k|^2\, d\mu_k \le C \|p_k\|^2, \qquad \forall p_k\in \mathcal
P_k.
\end{equation}
\end{definition}
In $M$ we consider the balls defined by any Riemannian metric.
\begin{prop}\label{PropCarleson}
 A sequence $\{\mu_k\}_k$ is a uniformly sequence of
Carleson measures if and only if there is a $C>0$ such that 
\[
 \mu_k(B(x,1/k))< C/k^n \text{ for all }x\in M, k\in\mathbb N.
\]
\end{prop}
\begin{proof}
 The necessity follows from testing \eqref{Carleson} against normalized
reproducing kernels. For any $x\in M$, let
$\kappa_{x,k}(y)=K_k(x,y)/\sqrt{K_k(x,x)}$. Then it is clear that for all $y\in
M$:
\[
 |\kappa_{x,k}(y)|\le |\langle
K_k(x,\cdot),K_k(y,\cdot)\rangle|/\sqrt{K_k(x,x)}\le \sqrt{K_k(y,y)}\simeq
k^{n/2}.
\]
and $|\kappa_{x,k}(x)|=\sqrt{K_k(x,x)}\simeq k^{n/2}$. On the other hand
since M is algebraic we have the classical Bernstein inequality, see
\cite{BosLevMilTay}:
\[
\sup_{M}\|\nabla_t \kappa_{x,k}\|\lesssim k \sup_{M} |\kappa_{x,k}|\simeq
k^{n/2+1} .
\]
This means that there is a $\delta>0$ such that 
\[
 |\kappa_{x,k}(y)| \gtrsim \kappa_{x,k}(x)\simeq k^{n/2}, \qquad \forall y\in
B_M(x,\delta/k).
\]
Therefore if we test \eqref{Carleson} with $\kappa_{x,k}$ we get that
$\mu_k(B_M(x,\delta/k))\lesssim 1/k^n$.

On the other direction for any $x\in M$ we consider a ball $B_X(x,r)$ the ball
in the complexified manifold $X$. By the submean property of the holomorphic
functions $f\in \mathcal H(X)$, we have that for $r\le r_0$, $|f(x)| \lesssim
\fint_{B_X(x,r)}|f(y)|\, dV_X(y)$. Thus
\[
 |p_k(x)|^2 \lesssim \fint_{B_X(x,1/k)}|p_k(y)|^2\, dV_X(y) \simeq
k^{2n}\int_{B_X(x,1/k)}|p_k(y)|^2\, dV_X(y).
\]
Finally, thanks to Theorem~\ref{tubular}
\[
\begin{split}
 \int_X |p_k(x)|^2\, d\mu_k(x)\lesssim\int_{x\in X}
k^{2n}\int_{B_X(x,1/k)}|p_k(y)|^2\, dV_X(y), d\mu_k(x)\lesssim\\
\int_{y\in X, d(y,M)<1/k} k^{2n}|p_k(y)|^2 \mu_k(B_X(y,1/k)\cap
X)dV_X(y)\lesssim\\
\int_{y\in X, d(y,M)<1/k}k^n|p_k(y)|^2dV_X(y)\lesssim \int_X |p_k(x)|^2\,
dV_M(x).
\end{split}
\]

\end{proof}

An immediate corollary is the description of the sequences that satisfy the
left hand side inequality of the sampling sequences, that is a 
Plancherel-Polya type inequality.

We say that a sequence of finite sets $\Lambda_k$ is \emph{uniformly separated} 
if and only if
there is an $\varepsilon>0$ such that $d(\lambda,\lambda')\ge \varepsilon/k$ for
all $\lambda\ne\lambda'$, $\lambda,\lambda'\in \Lambda_k$.
\begin{coro}[Plancherel-Polya type inequality]\label{PlancherelPolya}
 The sequence of finite sets  $\Lambda_k$ is a finite union of uniformly 
separated sequences if and
only if  there is a constant $C>0$ such that
 \[
  \frac 1{k^n}\sum_{\lambda\in \Lambda_k} |p(\lambda)|^2\le C \int_M |p|^2,
\quad \forall p\in\mathcal P_k.  
 \]
\end{coro}
\begin{proof}
 Apply Proposition~\ref{PropCarleson} to the measures $\mu_k=\frac
1{k^n}\sum_{\lambda\in\Lambda_k}\delta_{\lambda}$. This implies that the
Plancherel-Polya type inequality holds if and only if $\# \{\Lambda_k \cap
B(x,1/k)\}\le C$ uniformly in $x$ and $k$. That is, the sequence $\Lambda_k$ is 
a finite union of uniformly separated sequences of sets.
\end{proof}

Once we have a Bernstein type inequality the following Proposition is standard, 
see \cite[Proposition 5, p.~47]{Seip} and it allows to 
reduce our 
considerations to uniformly separated sequences.
\begin{prop}
 If $\Lambda_k$ is a sampling sequence then there is a
uniformly separated sequence of subsets $\Lambda'_k\subset \Lambda_k$ such that 
$\Lambda'_k$ is still a sampling sequence.
\end{prop}

It is also completely standard that 
\begin{prop}\label{sepint}
 If $\Lambda_k$ is an interpolating sequence then is is uniformly separated. 
\end{prop}

Theorem~\ref{tubular} can also be used to provide a sufficient condition that
assures the existence of sampling sequences. More precisely, we say
that the sequence $\Lambda_k$ is an $\epsilon$-net if it is uniformly separated 
and 
$1\le\sum_{\lambda\in \Lambda_k}\chi_{B(\lambda,\varepsilon/k)}(x)\le C_M$, for
all $x\in M$ and $k>0$, the constant $C_M$ depends on $M$ but
not on $\varepsilon$. By an application of the Vitali covering lemma it
is possible to construct $\epsilon$-nets for arbitrarily small $\varepsilon$.

\begin{prop}\label{nets}
 There is an $\epsilon_0$ such that any sequence $\Lambda_k$ that is an
$\epsilon$-net with $\epsilon<\epsilon_0$ is a sampling sequence.
\end{prop}
\begin{proof}
 Take one $\epsilon$-net $\Lambda_k=\Lambda_k(\varepsilon)$. Then
 \[
 \begin{split}
&\int_M |p_k|\,dV_M\le \sum_{\lambda\in
\Lambda_k}\int_{B_M(\lambda,\varepsilon/k)} |p_k|\,dV_M\le \\ 
&\le\sum_{\lambda\in \Lambda_k} |p_k(\lambda)||B_M(\lambda, \varepsilon/k)|+
\sum_{\lambda\in
\Lambda_k}\int_{B_M(\lambda,\varepsilon/k)}|p_k(x)-p_k(\lambda)|\,
dV_M\lesssim\\ 
&\lesssim\sum_{\lambda\in \Lambda_k}
\frac{\varepsilon^n}{k^n}|p_k(\lambda)|+ \sum_{\lambda\in \Lambda_k}
\frac{\varepsilon}{k}|\nabla_t p_k(\zeta_\lambda)||B_M(\lambda, \varepsilon/k)|,
\end{split}
\]
where $\zeta_\lambda\in \overline{B_M(\lambda,\varepsilon/k)} $ is such that
$|\nabla_t
p_k(\zeta_\lambda)|=\sup_{x\in B_M(\lambda,\varepsilon/k)}|\nabla_t p_k(x)|$. By
the Cauchy inequality, if we take a ball $B_X(\lambda,1/k)$ in the
complexification $X$ of $M$, we have
\[
 |\nabla_t p(\zeta_\lambda)| \lesssim k\fint_{B_X(\lambda,1/k)}|p_k|\, dV_X.
\]
Therefore,
\[
 \int_M |p_k|dV_M\lesssim \frac{\varepsilon^n}{k^n}\sum_{\lambda\in \Lambda_k}
|p_k(\lambda)| + \sum_{\lambda\in \Lambda_k} \varepsilon^{n+1} k^n
\int_{B_X(\lambda,1/k)}|p_k|\, dV_X.
\]
Since $\Lambda_k$ is an $\varepsilon$-net there are at most $C\varepsilon^{-n}$
points of $\Lambda_k$ in any given ball $B_X(x,1/k)$ of center $x\in U_{1/k}$.
Thus,
\[
 \int_M |p_k|dV_M\lesssim\frac{\varepsilon^n}{k^n}\sum_{\lambda\in \Lambda_k}
|p_k(\lambda)|+ \varepsilon k^n\int_{U_{1/k}} |p_k|\, dV_X.
\]
We use now Theorem~\ref{tubular} to control the right hand side integral. If we
take $\varepsilon$ small enough we can absorb the integral in the left hand
side and we get
\[
 \int_M |p_k|dV_M\lesssim\frac{1}{k^n}\sum_{\lambda\in \Lambda_k}
|p_k(\lambda)|
\]
The $L^\infty$ version: $\sup_M |p_k|\lesssim \sup_{\Lambda_k} |p_k|$ follows
immediately by the Bernstein inequality proved in \cite{BosLevMilTay}, if
$\varepsilon$ is small enough. By interpolation we get that for any
$q\in[1,\infty)$
\[
 \int_M |p_k|^q dV_M\lesssim\frac{1}{k^n}\sum_{\lambda\in \Lambda_k}
|p_k(\lambda)|^q
\]
The reverse inequality 
\[
\frac{1}{k^n}\sum_{\lambda\in \Lambda_k}
|p_k(\lambda)|^q\lesssim  \int_M |p_k|^q dV_M,
\]
follows from Corollary~\ref{PlancherelPolya} since $\Lambda_k$ is uniformly
separated.
\end{proof}

We can finish now the proof of Theorem~\ref{bernsteincomplet}
\begin{proof}
 We have already proved that the algebracity of $M$ is equivalent to
the Bernstein inequality, this is Theorem~\ref{Bernstein} and \ref{converse}.
Moreover Proposition~\ref{nets} proves that compact algebraic manifolds have
uniformly separated sampling sequences. So we only need to check that 
if there are
such sequences then $M$ is algebraic. This is proved in a similar way to
Theorem~\ref{converse}. We
denote  by $l_k=\#\Lambda_k$ as before. Define: $R_k: \mathcal P_k(M)\to \mathbb
R^{l_k}$ as 
\[
 R_k(p)(\lambda)= p(\lambda)\qquad \forall \lambda\in \Lambda_k.
\]
Clearly, since we have the sampling property, $R_k$ is one-to-one. Therefore
$\dim(P_k(M))\le l_k$. Moreover since $\Lambda_k$ is uniformly separated, then
$l_k\le k^{n}$. This implies that $M$ is algebraic.
\end{proof}

\subsection{A general off-diagonal estimate on the reproducing kernel }
\begin{thm}
\label{prop:off-diagonal bergman kernel}Let $M$ be an $n$-dimensional
affine real algebraic variety (possibly singular), $\mu_{k}$ a sequence
of non-degenerate finite measures on $M$ with support contained in a compact 
of $M$ and denote by $K_{k}(x,y)$
the reproducing kernel for the space $H_{k}(M)$, viewed as a subspace
of $L^{2}(M,\mu_{k})$. Then there exists a positive constant $C$
such that 
\[
\int_{M\times M}\frac{1}{k^{n}}\left|K_{k}(x,y)\right|^{2}d\mu_{k}(x)\otimes 
d\mu_{k}(y)\left|x-y\right|^{2}\leq C/k
\]
\end{thm}

\begin{remark}
 Observe that if we pick $\mu_k=e^{-k\phi}\mu$ the theorem covers the weighted 
setting as-well.
\end{remark}

\begin{proof}
Given a bounded function $f$ on $M$ we denote by $T_{f}$ be the
Toeplitz operator on $H_{k}(M)\cap L^{2}(M,\mu_{k})$ with symbol
$f$, i.e. $T_{f}:=\Pi_{k}\circ f\cdot$ where $\Pi_{k}$ denotes
the orthogonal projection from $L^{2}(M,\mu_{k})$ to $H_{k}(M)$,
i.e. $T_{f}$ is the Hermitian operator on $H_{k}(M)$ determined
by 
\[
\left\langle T_{f}p_{k},p_{k}\right\rangle _{L^{2}(M,\mu_{k})}=\left\langle 
fp_{k},p_{k}\right\rangle _{L^{2}(M,\mu_{k})}
\]
for any $p_{k}\in H_{k}(M)$. The following is essentially a
well-known formula
\[
\tr T_{f}^{2}-\tr T_{f^{2}}=\frac{1}{2}\int_{M\times 
M}\left|K_{k}(x,y)\right|^{2}d\mu_{k}(x)\otimes 
d\mu_{k}(y)\left(f(x)-f(y)\right)^{2}
\]
We provide nevertheless a proof for convenience of the reader:
\begin{claim}
 Let $H$ be a reproducing kernel Hilbert space with kernel $K$, then for any 
bounded symbol $f$ we have
 \[
  \iint |f(x)-f(y)|^2 |K(x,y)|^2 = \tr (2T_{|f|^2}-T_f\circ T_{\bar 
f}-T_{\bar f}\circ T_{f}).
 \]
\end{claim}
\begin{proof}
$K(x,y)=\sum_n f_n(x)\overline{f_n(y)}$ and
\[
 T_f(g)(x) = \int K(x,y)f(y)g(y).
\]
We compute the traces of $T_{|f|^2}$ and of $T_f\circ T_{\bar f}$.
\[
\begin{split}
 \tr(T_{|f|^2})&=\sum_n \langle f_n,T_{|f|^2}(f_n)\rangle=\\
 &\sum_n \int_x f_n(x) \overline{\int_y K(x,y) |f|^2(y)f_n(y)}=
 \iint |K(x,y)|^2\overline {|f|^2(y)}.
\end{split}
\]
Thus
\[
 \tr(T_{|f|^2})=\iint |K(x,y)|^2{|f(x)|^2}=\iint |K(x,y)|^2{|f(y)|^2}.
\]
Now
\[
\begin{split}
 \tr(T_f\circ T_{\bar f})&=\sum_n \int_x f_n(x)
 \overline{\int_y K(x,y)f(y) T_{\bar f}(f_n)(y)}=\\
 &=\sum_n \int_x\int_y f_n(x)\overline{K(x,y)}\overline{f(y)}
 \overline{\int_w K(y,w)\overline{f(w)}f_n(w)   }=\\
 &=\iiint K(x,w)\overline{K(y,w)}\overline{K(y,w)}\overline{f(y)}f(w)=
 \iint |K(y,w)|^2 \overline{f(y)}f(w).
\end{split}
\]
Similarly 
\[
 \tr(T_{\bar f}\circ T_{f})=\iint |K(y,w)|^2 f(y)\overline{f(w)}.
\]
\end{proof}
Now, setting $f:=x_{i}$ for a fixed index $i\in\{1,\ldots,m\}$ we note
that there exists a vector subspace $V_{k}$ in $H_{k}(M)$ such
that $\dim V_{k}=H_{k}(M)-O(k^{n-1})$ such that $T_{f}=f$ and $T_{f}^{2}=f^{2}$.
We can take $V_{k}$ to be the space spanned by the restrictions to
$M$ of all polynomials of total degree at most $k-1$, i.e. $H_{k-1}$. 
The dimension of $N_k=\dim(H_k)$ is the Hilbert polynomial of degree for 
$k\ge k_0$. Thus $N_k=d k^n + O(k^{n-1})$ where $d$ is the degree of the 
variety $M$ and $n$ is the dimension. 
In 
particular, denoting by $W_{k}$ the orthogonal complement
of $V_{k}$ in $H_{k}(M)\cap L^{2}(M,\mu_{k})$ then $\dim(W_k)=O(k^{n-1})$.  
Setting 
$A_{k}:=T_{f}^{2}-T_{f^{2}}$
gives $A_{k}=0$ on $V_{k}$ and hence
\[
\tr T_{f}^{2}-\tr T_{f^{2}}=0+\tr A_{k|_{W_{k}}}\leq 
Ck^{n-1}
\]
using that $\left\langle T_{f}p_{k},p_{k}\right\rangle 
_{L^{2}(M,\mu_{k})}\leq\sup|f|_{M}\left\langle p_{k},p_{k}\right\rangle 
_{L^{2}(M,\mu_{k})}$
and $\dim W_{k}=O(k^{n-1})$.
\end{proof}

\section{Sampling and interpolation of real orthogonal polynomials}

\subsection{Sampling polynomials in a real variety}
\begin{proof}[Proof of Theorem~\ref{general}]

We equip $M$
with the distance function $d$ induced by the Euclidean distance
in $\R^{m}$, i.e. $d(x,y):=|x-y|$. We recall that the corresponding
Wasserstein $L^{1}$-distance on the space $\mathcal{P}(M)$ of all
probability measures on $M$ is defined as
\[
 W(\mu,\sigma)=\inf_\rho \iint_{M\times M} d(x,y)\, d\rho(x,y),
\]
where the infimum is taken among all probability measures such that the first 
marginal of $\rho$ is $\mu$ and the second $\sigma$. The Wasserstein distance 
metrizes the weak-$*$ convergence.

We rely on the fact that 
\[
 \frac 1{N_k} B_k(x)\, d\mu_k(x)\to \nu(x),
\]
where the convergence is in the weak-$*$ topology, see \cite{bbw}. Thus the way 
to prove the 
inequality of the theorem is by proving that there are constants 
$\{c_\lambda\}_{\lambda\in 
S_k}$, $0\le c_\lambda<1$ such that 
\[
W(\sigma_k,\beta_k)\to 0
\]
where $\sigma_k=\frac 1{N_k}\sum_{\lambda\in \Lambda_k}c_\lambda 
\delta_\lambda$, 
$\beta_k=\frac 1{N_k} B_k(x)$. Instead of the standard Wasserstein distance we 
will use an alternative expression more convenient for our purpose that it is 
equivalent to it, see \cite{lo}:
\[
 W(\mu,\sigma)=\inf_\rho \iint_{M\times M} d(x,y)\, |d\rho(x,y)|,
\]
where the $\inf$ is taken among all complex measures $\rho$ such that the first 
marginal of $f$ is $\mu$ and the second $\sigma$. The difference is that 
$\rho$ is not necessarily positive and even if we don't require that $\sigma$ 
and $\nu$ are probability measures it still metrizes the weak-$*$ convergence. 
Any candidate $\rho$ with the right marginals is called a transport plan.

The transport plan $\rho_k$ that is convenient to estimate is:
\[
 \rho_k(x,y)=\frac 1{N_k} \sum_{\lambda\in \Lambda_k} \delta_\lambda(y) 
\times g_\lambda(x)\frac{K_k(\lambda,x)}{\sqrt{B_k(\lambda)}}\,d\mu_k(x).
\]
where $K_k(\lambda,x)$ is the reproducing kernel for $\lambda$ in the space 
$H_k$ 
and $\{g_\lambda\}_{\lambda\in S_k}$ is the canonical dual frame (see 
\cite{da}) to 
$\left\{\frac{K_k(\lambda,x)}{\sqrt{B_k(\lambda)}}\right\}_{\lambda\in 
\Lambda_k}$ in 
$H_k$. The latter is a frame because $\Lambda_k$ is sampling.

If we compute the marginals of $\rho_k$ we get on one hand:
\[
\sigma_k(y)= \frac 1{N_k}\sum_{\lambda_\in 
\Lambda_k}\frac{g_\lambda(\lambda)}{\sqrt{B_k(\lambda)}}\delta_\lambda(y).
\]
and the other marginal is given by
\[
d\beta_k(y)= \frac 1{N_k} 
\sum_{\lambda}g_\lambda(x)\frac{K_k(\lambda,x)}{\sqrt{B_k(\lambda)}}\, 
d\mu_k(x)=\frac 1{N_k} K_k(x,x)\, d\mu_k(x). 
\]
In the last equality we have used that $g_\lambda$ is a dual frame of the 
normalized reproducing kernels.

The fact that $\{g_\lambda\}$ it is the canonical dual frame to the 
normalized reproducing kernels allows us to conclude that 
$\frac{g_\lambda(\lambda)}{\sqrt{B_k(\lambda)}}=\langle g_\lambda(x), 
\frac{K_k(\lambda,x)}{\sqrt{B_k(\lambda)}}\rangle$ is positive and smaller than 
one.
This follows from the following well known fact:

\begin{claim} If $\{x_n\}_n$ is a frame in a Hilbert space $H$ and $\{y_n\}_n$ 
is the dual frame then $\langle x_n, y_n\rangle \in [0,1]$. 
\end{claim}
\begin{proof}
 Let $T$ be the frame operator, i.e: $T(x)=\sum \langle x,x_n\rangle  x_n$. 
Since $\{x_n\}_n$ is a frame then $T$ is bounded, self-adjoint and invertible. 
The definition of the dual frame is $T(y_n)=x_n$. For any vector $v\in H$ we 
have
\[
 v=T(T^{-1} v)=\sum_n \langle T^{-1}v,x_n\rangle  x_n.
\]
In particular
\[
x_k = \sum_n \langle y_k,x_n\rangle  x_n,
\]
and multiplying by $y_k$ at both sides we get
\[
 \langle x_k,y_k\rangle = \sum_n |\langle y_k,x_n\rangle |^2.
\]
Therefore $\langle x_k,y_k\rangle \ge 0$ and $\langle x_k,y_k\rangle>0$ unless 
$x_k=0$. Moreover,
\[
 \langle x_k,y_k\rangle - |\langle y_k,x_k\rangle |^2 = 
 \sum_{n\ne k} |\langle y_k,x_n\rangle |^2\ge 0.
\]
Thus 
\[
\langle x_k,y_k\rangle (1-  \langle x_k,y_k\rangle)\ge 0,
\]
therefore $\langle x_k,y_k\rangle\le 1$ too.

\end{proof}

Finally we need to estimate
\[
I= \iint_{M\times M} |x-y||d\rho_k|\le \frac 1{N_k}\sum_{\lambda\in \Lambda_k} 
\int_M |\lambda-x| 
|K_k(\lambda,x)|\frac{g_\lambda(x)}{\sqrt{B_k(\lambda)}}\, d\mu_k(x). 
\]
Since $\|g_\lambda\|_2 \simeq 1 $ we can estimate
\[
 I^2\lesssim \frac 1{N_k} \sum_{\lambda\in \Lambda_k} \int_M |\lambda-x|^2 
\frac{|K_k(\lambda,x)|^2}{B_k(\lambda)}
\]
We would like to use the sampling inequality \eqref{sampineq}, 
and obtain that
\begin{equation}\label{disccont}
\begin{split}  \frac 1{N_k} \sum_{\lambda\in \Lambda_k} \int_M
|\lambda-x|^2 
\frac{|K_k(\lambda,x)|^2}{B_k(\lambda)}\,d\mu_k(x)\le \\ 
\frac 1{N_k} 
\iint_{M\times 
M} |y-x|^2 
|K_k(y,x)|^2\, d\mu_k(x)d\mu_k(y).
\end{split}
\end{equation}
This we cannot do immediately because the polynomial $(x-y)K_k(y,x)$ (in the 
variable $y$) is of degree $k+1$ instead of $k$ as required in \eqref{sampineq}.

But we are assuming that $(\mu,\phi)$ define spaces with reproducing kernels of 
moderate growth. Thus $B_{k+1}\simeq B_k$ in $M$. Therefore if 
$\{\Lambda_k\}_k$ is sampling for $H_k(M)$ then $\{\Lambda_{k+1}\}_k$ is 
sampling for $H_k(M)$. Thus, it is harmless to assume that $\Lambda_k$ is 
sampling both for $H_k$ and for $H_{k+1}$ and we have established 
\eqref{disccont}.
Then, using Theorem~\ref{prop:off-diagonal bergman kernel}, we obtain
\[
 W(\sigma_k,\beta_k)=O(1/\sqrt{k}),
\]
as desired.
\end{proof}

\subsection{Interpolating polynomials in a real 
variety}\label{sec:interpolation}

\begin{definition}
A sequence $\Lambda_{k}$ of sets of points on $M$ is said to be
\emph{interpolating} for $H_{k}(M)$ if the family of normalized reproducing
kernels 
\[
\kappa_{\lambda}:=K_{k}(\cdot,\lambda)/\|K_{k}(\cdot,\lambda)\|
\]
for $\lambda\in\Lambda_{k}$, is a Riesz sequence in the Hilbert space
$H_{k}(M)$, i.e.:
\[
\frac 1{C} \sum_{\lambda\in\Lambda_k} |c_\lambda|^2 \le \|
\sum_{\lambda\in \Lambda_k} c_\lambda
\kappa_{\lambda}\|^2\le C   \sum_{\lambda\in\Lambda_k} |c_\lambda|^2,\quad
\forall \{c_\lambda\}_{\lambda\in\Lambda_k}\in \ell_2
\] 
where we will assume that $C$ can be taken independent of $k$.
\end{definition}

This property is equivalent to the Plancherel-Polya inequality:
\begin{equation}\label{plancherel-polya}
\sum_{\lambda\in\Lambda_k}\frac{|f(\lambda )|^2}{K_k(\lambda,\lambda)}
 \leq C\left\Vert f\right\Vert ^{2},\qquad \forall f\in H_k(M)
\end{equation}
and the \emph{interpolation property}:  for any sequence of sets of
values 
$\{c_\lambda^{(k)}\}_{\lambda\in \Lambda_k}$ there are functions $f_k\in H_k$
such 
that $f_k(\lambda^{(k)})=c_\lambda^{(k)}$ with 
\begin{equation}
\|f_k\|^2\le C \sum_{\lambda\in\Lambda_k}
\frac{|c_\lambda|^2}{K_k(\lambda,\lambda)},
\end{equation}
and again the constant $C$ should not depend on $k$.

The property that the collection $\{\kappa_\lambda\}_{\lambda\in \Lambda_k}$ is
a frame in $H_k(M)$ is a quantitative version of the fact that the normalized
reproducing kernels span the whole space and the property that they are a Riesz
sequence quantifies the fact that they are linearly independent.

\begin{proof}[Proof of Theorem~\ref{generalinterp}]
Let $F_k\subset H_k$ be the subspace spanned by 
\[
 \kappa_\lambda(x)=K_k(\lambda,x)/\sqrt{K_k(\lambda,\lambda)}\qquad \forall
\lambda \in \Lambda_k.
\] 
Denote by $g_\lambda$ the dual (biorthogonal) basis to $\kappa_\lambda$ in
$F_k$. We 
have clearly that 
\begin{itemize}
\item We can span any function in $F_k$ in terms of $\kappa_\lambda$, thus:
\[
 \sum_{\lambda\in \Lambda_k} \kappa_\lambda(x)g_\lambda(x)=\mathcal K_k(x,x),
\]
where $\mathcal K_k(x,y)$ is the reproducing 
kernel of the subspace $F_k$.  
\item The norm of $g_\lambda$ is uniformly bounded since $\kappa_\lambda$ was a 
uniform Riesz sequence.
\item $g_\lambda(\lambda)=\sqrt{K_k(\lambda,\lambda)}$. This is due to the 
biorthogonality and the reproducing property.
\end{itemize}
We are going to prove that the measure $\sigma_k=\frac 1{N_k}\sum_{\lambda\in 
\Lambda_k}\delta_\lambda$, 
and the measure $\beta_k=\frac 1{N_k} \mathcal K_k(x,x)d\mu(x)$ are very 
close to each other: 
$W(\sigma_k,\beta_k)\to 0$.
In this case then since $\mathcal K_k(x,x)\le K_k(x,x)$ and $\frac 1{N_k} 
B_k(x)\, 
d\mu\to d\nu$, where $\nu$ is the normalized equilibrium measure on $M$, then
$\limsup_k \sigma_k\le \nu$.

In order to prove that $W_k(\sigma_k,\beta_k)\to 0$ we use the transport plan:
\[
 \rho_k(x,y)=\frac 1{N_k} \sum_{\lambda\in \Lambda_k} \delta_\lambda(y) 
\times g_\lambda(x)\kappa_\lambda(x)\,d\mu(x)
\]
It has the right marginals, $\sigma_k$ and $\beta_k$ 
and we can estimate the integral in the same way as in the proof of 
Theorem~\ref{general}
\[
W(\sigma_k,\beta_k)\le
\iint_{M\times M} |x-y||d\rho_k|=O(1/\sqrt{k}).
\]
The only point that merits a clarification is that we need an inequality 
similar to \eqref{disccont}, i.e:
\[
\begin{split}  \frac 1{N_k} \sum_{\lambda\in \Lambda_k} \int_M
|\lambda-x|^2 
\frac{|K_k(\lambda,x)|^2}{ K_k(x,x)}\,d\mu(x)\le \\ 
\frac 1{N_k} 
\iint_{M\times 
M} |y-x|^2 
|K_k(y,x)|^2\, d\mu(x)d\mu(y).
\end{split}
\]
This time this is true because $\Lambda_k$ is a uniformly separated 
sequence by Proposition~\ref{sepint} and therefore, it is a Plancherel-Polya 
sequence, see Proposition~\ref{PlancherelPolya}.
\end{proof}

\subsection{Sampling in convex domains}
We proceed with the proof of Theorem~\ref{thm:samplingconvex}. The only part 
that we need to proof are the estimates for the reproducing kernel 
\eqref{kernelconvex}. If these are proved, then it follows that the measure has 
the Bernstein-Markov property \eqref{BM} and the kernel is of moderate growth
\eqref{MG}, thus we can 
apply Theorem~\ref{general}.

\begin{proof}
We start by the case when $\Omega=\B$ is the unit ball. We denote by 
$B_\Omega(x)$ the Bergman function which is the reproducing kernel $K_k(x,x)$ 
evaluated at the diagonal of the space of polynomials of total degree $k$ 
endowed with the $L^2$ norm with respect to the standard volume form. To get a 
lower bound for $B_\Omega$ we 
consider the cube $Q$ such that the ball is inside it and 
tangent to its faces. Clearly by the comparison principle of the Bergman 
functions $B_\B\ge B_Q$ and 
$B_Q(x)\simeq B_I(x_1)\cdots B_I(x_n)$ where $B_I$ is the one dimensional 
Bergman kernel associated to the interval. This is known 
to be, see \cite[p. 108]{Nevai}:
\[
 B_I(x) \simeq  \min\left( \frac{k}{\sqrt{d(x)}}, k^{2}\right).
\]
This implies that for points $x$ in the interval that joins the origin with the 
center of one of the faces of the cube $Q$ we have
\[
 B_Q(x)\simeq \min\left( \frac{k^n}{\sqrt{d(x)}}, k^{n+1}\right).
\]
Thus we have the lower bound for $B_{\B}$ that we wanted.
To get the upper bound we will work in dimension $n=2$ for simplicity but a 
similar argument works in any dimension. Observe that the space of polynomials 
of degree smaller or equal than $k$ is spanned by the functions
$\{\rho^j\cos^j(t),\rho^j\sin^j(t)\}_{j=0,\ldots,k}$ in polar coordinates in 
the interval $[0,1]\times [0,2\pi]$ with the measure $\rho d\rho$ in the first 
interval and $dt$ in the second. Consider now the space of functions 
$\tilde H_k$ in 
the product interval such that it is spanned by 
$\{\rho^{j}\cos^m(x),\rho^j\sin^m(x)\}_{j=0,\ldots,k\ m=0,\ldots,k}$. The space 
$\tilde H_k$ is bigger than the space of polynomials thus the Bergman function 
at the 
diagonal $B_{\tilde H_k}(x)\ge B_{\B,k}(x)$. But $B_{\tilde H_k}$ is easier to 
analyze 
because it is a product space of two one-dimensional spaces: The space of 
one dimensional polynomials of degree smaller $k$ with the norm $\rho\,d\rho$ 
in the interval $[0,1]$ and the space of trigonometric polynomials 
$\{\sin^j(x),\cos^j(x)\}_{j=0,\ldots k}$ with the measure $dx$ in $[0,2\pi]$. 
The Bergman function of $\tilde H_k$ is the product of the one-dimensional 
Bergman 
functions. The Bergman function corresponding to the trigonometric polynomials 
is constant by invariance under rotations and by dimensionality it must be 
$2k+1$. The space of polynomials in $\rho$ are a space of Jacobi polynomials 
and its Bergman function has been estimated, see \cite[p.~108]{Nevai}:
\[
B_J(x) \simeq  \min\left( \frac{k}{\sqrt{d(x)}}, k^{2}\right)\, \quad \forall 
x>1/2.
\]
Thus finally when $n=2$ we get
\[
 B_{\B}\lesssim \min\left( \frac{k^2}{\sqrt{d(x)}}, k^{3}\right).
\]
Similarly in higher dimension we get
\[
 B_{\B}\lesssim \min\left( \frac{k^n}{\sqrt{d(x)}}, k^{n+1}\right).
\]
Now for an arbitrary convex domain  there is an $r>0$ (small)  and an 
$R>0$ (big) that depend only on the
domain such that for any point $x$ in the boundary of the domain, there
is a ball $B(y,r)$ inside the domain tangent at $x$ and with center $y$ in the
normal direction to the boundary of the domain at $x$ and a cube
$Q(R)$ tangent to the domain at $x$ in the middle of a face of the
cube and such that the domain is contained in the cube. Again by the comparison 
principle of the Bergman function we get
\[
 \min\left( \frac{k^n}{\sqrt{d(x)}}, k^{n+1}\right)\lesssim
 B_{Q(R)}(x)\lesssim B_\Omega(x)\lesssim B_{B(y,r)}(x)\lesssim
 \min\left( \frac{k^n}{\sqrt{d(x)}}, k^{n+1}\right).
\]
\end{proof}

\subsection{Existence of interpolating and sampling sequences in the 
one-dimensional setting}
We conclude the paper by recalling some classical facts which are special for 
the one dimensional setting.

Let $\mu$ be a finite measure on $\R$ with compact set $K$ and
assume that $\mu$ has the Bernstein-Markov property with respect to $K$. By
the classical Christoffel-Darboux formula there exists
constants $a_{k+1}$ such that

\[
K_{k}(x,y)=a_{k+1}\frac{q_{k+1}(x)q_{k}(y)-q_{k}(x)q_{k+1}(y)}{x-y}
\]
where $q_{k+1}$ is the $k$ th orthogonal polynomial (with respect to $\mu)$.
Let $\Lambda_{k}:=\{x_{j}^{(k)}\}$ be the set of $k+1$ zeros of
$q_{k}$ (which by classical results are indeed all distinct and contained
in the support $K$ of $\mu)$. Then $K_{k}(x_{i}^{(k)},x_{j}^{(k)})=0$
if $i\neq j$, as follows immediately from the Christoffel-Darboux
formula. Hence, normalizing $K_{k}(\cdot,x_{j}^{(k)})$ yields an
orthonormal base in $H_{k}(M,\mu)$ and as a consequence the following
``sampling equality'' holds for any $p_{k}\in H_{k}(M):$ 
\[
\int_{M}\left|p_{k}\right|^{2}d\mu=\sum_{i}\frac{1}{B_{k}(x)}\left|p_{k}(x_{i}^{
(k)})\right|^{2},
\]
 and in particular the sequence $\Lambda_{k}$ is both sampling and 
interpolation. Finally,
recall that by classical results the normalized Dirac measure $\delta_{k}$
on the zeros $\Lambda_{k}$ has the same weak limit points as $B_{k}/(k+1)\mu$.
In particular, if $\mu$ has the Bernstein-Markov property, then 
$\frac 1{k}\sum\delta_{k}\rightarrow\mu_{eq}$,
which is thus consistent with Theorem~\ref{general} and 
Theorem~\ref{generalinterp}.

\end{document}